\documentclass[12pt]{article}

\usepackage[margin=1.3in]{geometry}

\usepackage{amssymb,amsmath}
\usepackage{amsthm}
\usepackage{hyperref}
\usepackage{mathrsfs}
\usepackage{textcomp}
\hypersetup{
    colorlinks,%
    citecolor=black,%
    filecolor=black,%
    linkcolor=black,%
    urlcolor=black
}

\usepackage{verbatim}

\usepackage{sectsty}

\makeatletter

\newdimen\bibspace
\renewenvironment{thebibliography}[1]{%
 \section*{\refname 
       \@mkboth{\MakeUppercase\refname}{\MakeUppercase\refname}}%
     \list{\@biblabel{\@arabic\c@enumiv}}%
          {\settowidth\labelwidth{\@biblabel{#1}}%
           \leftmargin\labelwidth
           \advance\leftmargin\labelsep
           \itemsep\bibspace
           \parsep\z@skip     %
           \@openbib@code
           \usecounter{enumiv}%
           \let\p@enumiv\@empty
           \renewcommand\theenumiv{\@arabic\c@enumiv}}%
     \sloppy\clubpenalty4000\widowpenalty4000%
     \sfcode`\.\@m}
    {\def\@noitemerr
      {\@latex@warning{Empty `thebibliography' environment}}%
     \endlist}

\makeatother

\makeatletter

\newtheorem{thm}{Theorem}[section]

\newtheorem{prop}[thm]{Proposition}

\def\XXint#1#2#3{{\setbox0=\hbox{$#1{#2#3}{\int}$}
  \vcenter{\hbox{$#2#3$}}\kern-.5\wd0}}

\newcommand{\al}{\alpha}                \newcommand{\lda}{\lambda}
\newcommand{\om}{\Omega}                \newcommand{\pa}{\partial}
\newcommand{\va}{\varepsilon}           \newcommand{\ud}{\mathrm{d}}
\newcommand{\be}{\begin{equation}}      \newcommand{\ee}{\end{equation}}
\newcommand{\w}{\omega}                 
              
\newcommand{\R}{\mathbb{R}}

\begin{document}

\title{\textbf{On a conformally invariant integral equation involving Poisson kernel}
\bigskip}

\author{Jingang Xiong\footnote{Supported in part by NSFC 11501034, a key project of NSFC 11631002 and NSFC 11571019.}}

\date{ }

\maketitle

\begin{abstract} We study a prescribing functions problem of a conformally invariant integral  equation involving Poisson kernel on the unit ball. This integral equation is not the dual of any standard type of PDE.  As in Nirenberg problem, there exists a Kazdan-Warner type obstruction to existence of solutions. We prove existence in the antipodal symmetry functions class.
\end{abstract}

\section{Introduction}

Poisson integral and Riesz potential are basic objects in the singular integral theory; see Stein \cite{S}. Riesz potential is the dual of (fractional) Poisson equations. In \cite{JLX3}, Jin-Li-Xiong developed a blow up analysis procedure for critical nonlinear integral equations involving Riesez kernel and established a unified approach to the Nirenberg problem and its generalizations. The method is flexible; see Li-Xiong \cite{LX} for its application to compactness of  fourth order constant $Q$-curvature metrics. In this paper, we extend some analysis further to a natural critical nonlinear integral equations involving Poisson kernel.

Let $B_1$ be the unit ball in $\R^n$, $n\ge 2$. For each $v\in L^p(\pa B_1)$, $p\ge 1$, the Poisson integral of $v$ is defined by
\be\label{eq:poissonformula}
\mathcal{P} v(\xi)= \int_{\pa B_1}P(\eta,\xi) v(\eta)\,\ud s_\eta \quad \mbox{for }\xi \in B_1,
\ee
where $
P(\eta,\xi)= \frac{1-|\xi|^2}{n \w_n}\frac{1}{|\xi-\eta|^n}$ is the Poisson kernel
and $\w_n$ is the volume of the unit ball. Then $\mathcal{P} v$ is a harmonic function in $B_1$. If $n=2$, a classical inequality of Carleman \cite{Car} asserts that
\[
\int_{B_1}e^{2\mathcal{P} v}\,\ud \xi \le \frac{1}{4\pi}(\int_{\pa B_1} e^v\,\ud s)^2
\]
and the equality holds if and only if $v=c$ or $v=-2\ln |\xi-\xi_0|+c$ for some constant $c$ and $\xi_0\in \R^2\setminus \bar B_1$. If $n\ge 3$,  Hang-Wang-Yan \cite{HWY} proved that
\be\label{eq:main-ineq}
\|\mathcal{P} v\|_{L^{\frac{2n}{n-2}(B_1)}} \le S(n)\|v\|_{L^{\frac{2(n-1)}{n-2}}(\pa B_1)},
\ee
where $S(n)=n^{-\frac{n-2}{2(n-1)}}\w_n^{-\frac{n-2}{2n(n-1)}}$ and the equality holds if and only if $v=1$ up to a conformal transform on the unit sphere $\pa B_1$. In \cite{HWY2}, they studied \eqref{eq:main-ineq} on Riemannian manifolds. See also the recent paper Dou-Guo-Zhu \cite{DGZ} and references therein for other related results.
Motivated by the Nirenberg problem,  starting from this paper we study positive solutions of the Euler-Larange equation of the functional
\[
I[v]= \frac{ \int_{B_1} |\mathcal{P} v|^{\frac{2n}{n-2}} \,\ud \xi }{(\int_{\pa B_1} K|v|^{\frac{2(n-1)}{n-2}}\,\ud s)^{\frac{n}{n-1}}},
\]
where  $v\in L^{\frac{2(n-1)}{n-2}}(\pa B_1)$ is not zero and $K>0$ is a given continuous function. Namely,
\be \label{eq:problem}
K(\eta) v(\eta)^{\frac{n}{n-2}}= \int_{B_1} P(\eta,\xi) \mathcal{P} v(\xi)^{\frac{n+2}{n-2}}\,\ud \xi, \quad v>0 \quad \mbox{on }\pa B_1.
\ee
This equation is critical, conformally invariant and not always solvable. Indeed, a Kazdan-Warner type necessary condition was derived in \cite{HWY2}: For any conformal Killing vector field $X$ on $\pa B_1$, endowed with the induced metric from $\R^n$,
\be\label{eq:KW}
\int_{\pa B_1} (\nabla_X K)  v^{\frac{2(n-1)}{n-2}}\,\ud s=0
\ee
holds for any solution $v$ of \eqref{eq:problem}. For example, if $K=\xi_n+2$, there is no solution of \eqref{eq:problem}.

\begin{thm}\label{thm:A} Let $n\ge 3$ and $K\in C^1(\pa B_1)$ be a positive function satisfying $K(\xi)=K(-\xi)$.  For every $q>n-1$, there exists a constant $\delta>0$, depending only on $n$ and $q$, such that if for  a minimal point $\xi_1$ of $K$ there holds $K(\xi)-K(\xi_1)\le \delta |\xi-\xi_1|^q$ for all $\xi\in \pa B_1$, then equation \eqref{eq:problem} has at least one positive solution.
\end{thm}

The analogue of Theorem \ref{thm:A} for Nirenberg problem was established by Escobar-Schoen \cite{ES}. See Jin-Li-Xiong \cite{JLX3} and references therein for generalized Nirenberg problems.
We prove Theorem \ref{thm:A} via subcritical approximation approach, which contains two steps. The first shows that if the supremum of $I[\cdot]$ is greater than some threshold, then maximizers exist.
Here we use a blow up analysis argument for integral equations, which was introduced by Jin-Li-Xiong \cite{JLX3}.  Our current equation has a stronger nonlocal feature.
New ingredients, such as boundary Harnack inequality,  are incorporated in the proofs.

The second step verifies the strict inequality.
Again due to the strong nonlocality, we introduce a trial  function by gluing two bubbles along the equator of the sphere. This is different from the Nirenberg problem case; see \cite{JLX1, JLX3} and references therein.  These two bubbles do not affect each other in the boundary $L^{\frac{2(n-1)}{n-2}}$ norm, but they do in the interior thanks to the harmonic extension. In particular, in the interior our trial function will be of the ``\emph{a bubble plus a positive harmonic function of a linear growth}" structure locally.   Such structure  was used by the author \cite{X} to study boundary isolated singularity in a different context; see the proof of Proposition 4.2 in that paper. See also Jin-Xiong \cite{JX}.   In \cite{SX}, Sun-Xiong proved ``\emph{bubbles plus polynomials}" type classification theorems of higher order boundary conformally invariant problems.

In the future work, we will study existence and compactness of solutions beyond the antipodal symmetry functions class. We will also study the exponential nonlinearity problem of dimension two as the classical work Chang-Yang \cite{CY1, CY2} did.

At the end of this section, we note that the Poisson kernel on the upper half space (see section \ref{sec:blowup}) coincides with the heat kernel of $\pa_t+(-\Delta)^{1/2}$, see Blumenthal-Getoor \cite{BG}. Hence, our problem can also be interpreted  through the $1/2$ heat kernel. Within this in mind, one may draw an analogy to the studies of maximizers for the Strichartz inequality and Stein-Tomas inequality; see Foschi \cite{F}, Christ-Shao \cite{CS1,CS2}, Frank-Lieb-Sabin \cite{FLS} and references therein.

\bigskip

\noindent \textbf{Acknowledgments:} The author thanks Tianling Jin for valuable discussions.

\section{A blow up analysis procedure}
\label{sec:blowup}

We denote $x=(x',x_n)$, $y= (y',y_n)$ as points in $\R^n$,  $B_R(x)$ as the open ball of $\R^n$ centered as $x$ with radius $R$, and $B'_{R}(x')$ as the open ball in $\R^{n-1}$ centered as $x'$ with radius $R$.

Let  $F:\R^{n}_+\to B_1$ be the Mobius transformation given by
\[
F(x)= \frac{2(x+ e_n)}{|x+ e_n|^2} -e_n,
\]
where $e_n=(0,\dots, 0,1)$. For $x=(x',0)\in \pa \R^{n}_+$, we see that
\[
F(x)=  (\frac{2x'}{|x'|^2+1},\frac{1-|x'|^2}{|x'|^2+1})\in \pa B_1
\]
is the inverse of the stereographic projection. For $v\in L^{\frac{2(n-1)}{n-2}}(\pa B_1)$,  let
\be\label{eq:transform}
u(x')=(\frac{\sqrt{2}}{|x+e_n|})^{n-2} v(F(x)) \quad  \mbox{for }x=(x',0).
\ee
For saving notations, we still use $P(\cdot,\cdot)$ to denote the Poisson kernel on the upper half space and
\[
\mathcal{P}u(x)=\int_{\R^{n-1}}  P(y',x) u(y')\,\ud y' =\frac{2}{n\w_n} \int_{\R^{n-1}}  \frac{x_n}{(|x'-y'|^2+x_n^2)^{\frac{n}{2}}} u(y')\,\ud y'.
\]
It is easy to check that
\be\label{eq:transform-1}
\mathcal{P} u(x)= (\frac{\sqrt{2}}{|x+e_n|})^{n-2} (\mathcal{P}v)(F(x)) .
\ee
We will use the fact
\be  \label{eq:kernelsmooth}
|\nabla _{x'}^k P(y',x)|= |\nabla^k _{y'}  P(y',x)| \le C(k) x_n(|x'-y|^2+x_n^2)^{-\frac{n+k}{2}}
\ee
for $x'\neq y', k=1,\dots, $ to obtain regularity.

The main result of this section is the blowing up a bubble result as follows.

\begin{thm}\label{thm:bp procedure} Let $ \frac{n}{n-2}\le p_i<\frac{n+2}{n-2}$ be a sequence numbers with $\lim_{i\to \infty}p_i= \frac{n}{n-2}$, and $K_i\in C^1(B_1')$ be a sequence of  positive  functions satisfying
\[
K_i\ge \frac{1}{c_0}, \quad \|K_i\|_{C^1(B_1')}\le c_0
\] for some constant $c_0\ge 1$ independent of $i$. Suppose that $u_i\in C^0(\R^{n-1})$ is a sequence of nonnegative solutions of
\be\label{eq:2.1}
K_i(x') u_i(x')^{p_i}=\int_{\R^{n}_+} P(x',y)\mathcal{P}u_i(y)^{\frac{n+2}{n-2}}\,\ud y \quad \mbox{for }x'\in  B_1'
\ee
and $
u_i(0)\to \infty $ as $i\to \infty$. Suppose that  $R_i u_i(0)^{p_i-\frac{n+2}{n-2}}\to 0$ for some $R_i\to \infty$ and
\[
u_i(x')\le bu_i(0) \quad \mbox{for }|x'|<R_i u_i(0)^{p_i-\frac{n+2}{n-2}},
\] where $b>0$ is independent of $i$. Then, after passing to a subsequence, we have
\be\label{eq:convergence}
\phi_i(x'):=\frac{1}{u_i(0)} u_i(u_i(0)^{p_i-\frac{n+2}{n-2}}x') \to \phi(x') \quad \mbox{in }C_{loc}^{1/2}(\R^{n-1}),
\ee
where $\phi> 0$ satisfies
\be \label{eq:limit}
K \phi(x')^{\frac{n}{n-2}}=\int_{\R^{n}_+} P(x',y)\mathcal{P} \phi(y)^{\frac{n+2}{n-2}}\,\ud y \quad \mbox{for }x'\in  \R^{n-1}
\ee
and $K=\lim_{i\to \infty} K_i(0)$ along the subsequence.
\end{thm}

Solutions of \eqref{eq:limit} in $L^{\frac{2(n-1)}{n-2}}_{loc}(\R^{n-1})$ were classified in \cite{HWY}, which are $(1+|x|^2)^{-\frac{n-2}{2}}$ upon multiplying, translating and
scaling.

\begin{proof} Note that
\be\label{eq:scale}
H_i(x')\phi_i(x')^{p_i}=\int_{\R^{n}_+} P(x',y)\mathcal{P}\phi_i(y)^{\frac{n+2}{n-2}}\,\ud y \quad \mbox{for }|x'|<R_i,
\ee
where $H_i(x')=K_i(u_i(0)^{p_i-\frac{n+2}{n-2}}x') $. By the assumption, we have
\be\label{eq:bound}
0\le \phi_i(x') \le b \quad \mbox{for }|x'|<R_i.
\ee

\textbf{Step 1.} Estimates of $\phi_i$ and  convergence.

For any fixed $0<R<R_i/2$, define
\[
\Phi_i'= \mathcal{P}(\chi_{B'_{R}} \phi_i) \quad \mbox{and} \quad \Phi_i''= \mathcal{P}((1-\chi_{B'_{R}} ) \phi_i),
\]
where $\chi_\om$ is the characterization function of the set $\om$. Then $\mathcal{P} \phi_i=\Phi_i'+\Phi_i'' $.
Since the Poisson kernel is nonnegative, by \eqref{eq:bound} we have
\be
\label{eq:phi'}
 0\le \Phi_i'(y) \le  b.
\ee Since $K_i\le c_0$ and \eqref{eq:bound}, by \eqref{eq:scale}  we have for any $|x'|<R_i$,
\begin{align*}
c_0 b^{p_i}& \ge  \int_{B_{1/2}(x',e_n)} P(x',y) \mathcal{P} \phi_i(y)^{\frac{n+2}{n-2}}\,\ud y \\ &
\ge   \frac1C \int_{B_{1/2}(x',e_n)} \mathcal{P} \phi_i(y)^{\frac{n+2}{n-2}}\,\ud y  \ge
\frac{1}{C}\mathcal{P} \phi_i(\bar y) ^{\frac{n+2}{n-2}}
\end{align*}
for some $\bar y\in  B_{1/2}(x',e_n)$, where we used the  mean value theorem in the last inequality and $C>0$ depends only on $n$. It follows that
\be \label{eq:2.3}
\Phi_i''(\bar y)\le \mathcal{P} \phi_i(\bar y)\le C b^{\frac{p_i(n-2)}{n+2}}.
\ee  By the definition of $\Phi_i''(\bar y)$, we immediately see the boundary Harnack inequality
\be \label{eq:rest-1}
\frac{\Phi_i''(y)}{y_n}\le C\frac{\Phi_i''(\bar y)}{\bar y_n}  \quad \mbox{for } y\in B'_{1}(x')\times (0,2], ~|x'|<R-1,
\ee
where $C>0$ depends only on $n$.  Combining \eqref{eq:phi'}, \eqref{eq:2.3} and \eqref{eq:rest-1} together, we have
\[
\mathcal{P} \phi_i(y) \le C \quad \mbox{for every } y\in B'_{R-1}\times (0,1].
\]

Using the above estimate, by direct computations we have
\begin{align*}
&\| \int_{ B'_{R-1}\times (0,1]} P(\cdot,y)\mathcal{P} \phi_i (y)^{\frac{n+2}{n-2}}\,\ud y \|_{C^\al(B'_{R-2})}\\&
\le C  \int_{ B'_{R-1}\times (0,1]} (|x'-y'|^2+y_n^2)^{\frac{n-1+\al }{2}}\ud y  \le C(n,b,\al,R)
\end{align*}
for any $\al\in (0,1)$. On the other hand, for $|x'|<R-2$, by \eqref{eq:kernelsmooth} we have
\begin{align*}
|\nabla_{x'} (\int_{\R^{n}_+\setminus B'_{R-1}\times (0,1]} P(x',y)\mathcal{P} \phi_i(y)^{\frac{n+2}{n-2}}\,\ud y)|& \le C\int_{\R^{n}_+\setminus B'_{R-1}\times (0,1]} P(x',y)\mathcal{P} \phi_i(y)^{\frac{n+2}{n-2}}\,\ud y\\&
\le C (H_i(x') \phi_i(x'))^{p_i} \le C b^{p_i}.
\end{align*}
where $C>0$ depends only on $n$ and $c_0$.  Combining the above two estimates and  using  \eqref{eq:scale} we conclude that with $\al=3/4$  \be \label{eq:powerholder}
\|\phi_i^{p_i}\|_{C^{3/4}(B'_{R-1})} \le C(n,b,c_0,R).
\ee

Since $\phi_i(0)^{p_i}=1$, by \eqref{eq:powerholder} one can find  $\delta>0$, depending only on $n$, $b$ and $c_0$, such that $\phi_i(x')^{p_i}\ge 1/2$ for all $|x'|<\delta$. Hence,
\[
\mathcal{P}\phi_i(y)\ge y_n c(n)\int_{B'_\delta} \frac{1}{(|x'-y'|^2+y_n^2)^{n/2}} 2^{-1/p_i}\,\ud x' \ge \frac{1}{C(n,b,c_0)} \frac{y_n}{(1+|y|)^n}.
\]
for some $C(n,b)>0$. Inserting the above estimate into \eqref{eq:scale}, we see that for any $0<|x'|<R-1$
\[
\phi_i(x')^{p_i} \ge \frac{1}{C(n,b,c_0,R)}>0.
\]
It follows from \eqref{eq:powerholder} that
\be \label{eq:compactnessatboundary}
\|\phi_i\|_{C^{3/4}(B'_{R-2})} \le C(n,b,c_0,R).
\ee  Therefore,  \eqref{eq:convergence} follows.

\textbf{Step 2.} $\mathcal{P}\phi_i$ and the equation of $\phi_i$ convergence.

The difficulty arises because there is no information about the behavior of  $\phi_i$ in the complement of $B_{R_i}'$.  Here we adapt some idea from \cite{JLX3} by using monotonicity. For any $0<R<R_i/2$, we write equation \eqref{eq:scale} as
\be \label{eq:truncat-1}
H_i(x')\phi_i(x')^{p_i}=\int_{B_R^+} P(y,x')\mathcal{P}\phi_i(y)^{\frac{n+2}{n-2}}\,\ud y+h_i(R,x'),
\ee
where
\[
h_i(R,x') =\int_{\R^n_+\setminus B_R^+} P(y,x')\mathcal{P} \phi_i(y)^{\frac{n+2}{n-2}}\,\ud y.
\]
By \eqref{eq:kernelsmooth}, for any $|x'|<R-1$ we have $|\nabla h_i(R,x')|\le C h_i(R, x')\le Cb^{p_i}$ for some $C>0$ depending only on $n$ and $R$. Therefore, subject to subsequence $h_i(R,x')\to h(R,x')$ for some nonnegative function $h\in C^1(B_{R-1})$.

Similar as in step 1, we split $\mathcal{P}\phi_i$ as two parts $\Phi_i'$ and $\Phi_i''$ with $R$ replaced by $R+10$. By \eqref{eq:compactnessatboundary} with $R-2$ replaced by $R+8$ and elementary estimates of Poisson integral,  we have   $\|\Phi_i'\|_{C^{3/4}(B_R^+)}\le C(n,b,R)$. While
$\|\Phi_i''\|_{C^{3/4}(B_R^+)}\le C(n,b,R)$ follows from \eqref{eq:2.3}, \eqref{eq:rest-1} and interior estimates for harmonic functions. Therefore,  subject to a subsequence,
\[
\mathcal{P} \phi_i \to \tilde \Phi \quad \mbox{in }C^{1/2}_{loc}(\bar \R^n_+)
\]
for some $\tilde \Phi\ge 0$ satisfies
\[
-\Delta \tilde \Phi=0 \quad \mbox{in }\R^n_+ \quad \mbox{and} \quad \tilde \Phi=\phi \quad \mbox{on }\pa \R^n_+.
\]
Since $0\le \phi\le b$, $\mathcal{P}\phi$ is bounded in $\R^n_+$. Hence, $\tilde \Phi-\mathcal{P}\phi$ is harmonic function bounded from below in $\R^n_+$ and satisfies the homogenous Dirichlet boundary condition. It follows from the Liouville theorem on the half space, see, e.g., Sun-Xiong \cite{SX}, that
\be\label{eq:13}
\tilde \Phi(x)=\mathcal{P} \phi(x)+a x_n \quad \mbox{for some constant }a\ge 0.
\ee
Sending $i\to \infty$ in \eqref{eq:truncat-1}, we have
\be\label{eq:truncat-2}
K\phi(x')^{\frac{n}{n-2}}= \int_{B_R^+} P(y,x')\tilde \Phi(y)^{\frac{n+2}{n-2}}\,\ud y+h(R,x')
\ee
If $a>0$ in \eqref{eq:13}, sending $R\to \infty$ we see that
\[
K\phi(0)^{\frac{n}{n-2}}\ge \int_{B_R^+} P(y,0)\tilde \Phi(y)^{\frac{n+2}{n-2}}\,\ud y \to \infty.
\]
This is impossible. Hence, $a=0$ and $\tilde \Phi=\mathcal{P} \phi(x)$. By \eqref{eq:truncat-2}, $h(R,x')$ is decreasing with respect to $R$. Note that for $R>>|x'|$,
\begin{align*}
\frac{R^{n}}{(R+|x|)^{n}} h_i(R,0)&\le h_i(R, x') \\&
= \int_{\R^n_+\setminus B_R^+} \frac{|y|^n}{(|y'-x'|^2+y_n^2)^{\frac{n}{2}}} \frac{y_n}{|y|^n} \mathcal{P} \phi_i (y)^{\frac{n+2}{n-2}}\,\ud y\\&
\le \frac{R^{n}}{(R-|x|)^{n}} h_i(R,0).
\end{align*}
It follows that
\[
\lim_{R\to \infty} h(R,x')=\lim_{R\to \infty} h(R,0)=:c_1\ge 0.
\]
Sending $R$ to $\infty$ in \eqref{eq:truncat-2}, by Lebesgue's monotone convergence theorem we have
\[
K\phi(x')^{\frac{n}{n-2}}= \int_{\R^n_+} P(y,x') \mathcal{P} \phi (y)^{\frac{n+2}{n-2}}\,\ud y+c_1.
\]
If $c_1>0$, then $\phi\ge \frac{c_1}{c_0}>0$ and thus $\mathcal{P} \phi\ge \frac{c_1}{c_0}.$ This is impossible, otherwise the right hand side integration is infinity. Hence $c_1=0$.

Therefore, we complete the proof.

\end{proof}

\section{A variational problem}

Let $K\in C^1(\pa B_1)$ be a positive function satisfying $K(\xi)=K(-\xi)$, and $L^p_{as}(\pa B_1)\subset L^p(\pa B_1)$, $p\ge 1$,  be the set of antipodally symmetric functions. For $p\ge \frac{n}{n-2}$, define
\[
\lda_{as,p}(K)=\sup \left\{\int_{B_1} |\mathcal{P} v|^{\frac{2n}{n-2}}\,\ud \xi: v\in L^{p+1}_{as}(\pa B_1) \mbox{ with } \int_{\pa B_1} K|v|^{p+1}\,\ud s=1 \right \}.
\]
Denote $\lda_{as,\frac{n}{n-2}}=\lda_{as}$ for brevity.
\begin{prop}\label{prop:existence} If
\be\label{eq:strictineq}
\lda_{as}(K) >\frac{S(n)^{\frac{2n}{n-2}}}{(\min_{\pa B_1} K)^{\frac{n}{n-1}}2^{1/(n-1)}},
\ee
then $\lda_{as}(K)  $ is achieved.
\end{prop}

\begin{proof} We claim that $\liminf_{p\searrow \frac{n}{n-2}} \lda_{as,p}(K)\ge \lda_{as}(K).$

Indeed, for any $\va>0$, by the definition of $\lda_{as}(K)$  one can find a function $v\in L_{as}^{\infty}(\pa B_1)$ such that
\[
\int_{B_1}|\mathcal{P} v |^{\frac{2n}{n-2}}\,\ud \xi>\lda_{as}(K)-\va \quad \mbox{and}\quad \int_{\pa B_1} K |v|^{\frac{2(n-1)}{n-2}}\,\ud s=1.
\]
Let $V_p:=\int_{\pa B_1} K |v|^{p+1}\,\ud s$. Since $\lim_{p\to\frac{n}{n-2} } V_p=\int_{\pa B_1} K |v|^{\frac{2(n-1)}{n-2}}\,\ud s=1$, we have, for $p$ close to $\frac{n}{n-2}$,
\[
\lda_{as,p}(K)\ge \int_{B_1}|\mathcal{P}(\frac{v}{V_p^{1/(p+1)}})|^{\frac{2n}{n-2}}\,\ud \xi\ge \lda_{as}(K)-2\va.
\]
By the arbitrary choice of $\va$, the claim follows.

By the above claim, one can seek $p_i \searrow  \frac{n}{n-2}$ as $i\to\infty$ such that $\lda_{as, p_i}(K) \to \lda\ge \lda_{as}(K)$. Since $K\in C^1(\pa B_1)$ and $K$ is positive, it follows from the compact embedding result Corollary 2.2 of \cite{HWY} that for $p_i>\frac{n}{n-2}$,  $\lda_{as,p_i}(K)$ is achieved, say, by $v_i$. Since $|\mathcal{P}v_i|\le \mathcal{P}|v_i| $, we may assume $v_i$ is nonnegative. Noticing that $\|v_i\|_{L^{p_i+1}(\pa B_1)}^{p_i+1}\le 1/\min_{\pa B_1} K$, by \eqref{eq:main-ineq} we have $\|\mathcal{P}v_i\|_{L^{\frac{2n}{n-2}}( B_1)}\le C$ for some $C$ independent of $i$.  It is easy to see that $v_i$ satisfies the Euler-Lagrange equation
\be \label{eq:22}
\lda_{as,p_i}(K) K(\xi)v_i(\xi)^{p_i}=  \int_{B_1} P(\xi,\eta)\mathcal{P}v_i(\eta)^{\frac{n+2}{n-2}}\,\ud \eta  \quad\forall~ \xi\in \pa B_1.
\ee
Hence, subject to a subsequence,
\begin{align*}
v_i \rightharpoonup v& \quad \mbox{weakly in } L^{\frac{2(n-1)}{n-2}}(\pa B_1)\\
\mathcal{P}v_i \rightharpoonup V& \quad \mbox{weakly in } L^{\frac{2n}{n-2}}(B_1)
\end{align*}
for some nonnegative function  $v\in L^{\frac{2(n-1)}{n-2}}(\pa B_1)$ and $V\in L^{\frac{2n}{n-2}}(B_1)$. By the compact embedding again, $V= \mathcal{P}v$. Hence, $v$ satisfies
\be
\lda K(\xi) v(\xi)^{\frac{n}{n-2}}= \int_{B_1} P(\xi,\eta)\mathcal{P}v(\eta)^{\frac{n+2}{n-2}}\,\ud \eta.
\ee
It follows that either $v\equiv 0$ or $v>0$. If the later happens, then $\lda=\lda_{as}(K)$ and we are done. Suppose now $v\equiv 0$.

By Proposition 5.2 of \cite{HWY} and equation  \eqref{eq:22}, $v_i\in C(\pa B_1)$. By standard arguments (see the proof of Theorem \ref{thm:bp procedure}),  we have $v_i\in C^{\al}(\pa B_1)$ for any $0<\al<1$. Since $v=0$,  we must have
$v_i(\xi_i)=\max_{\pa B_1} v_i \to \infty$ as $i\to \infty$. We may assume $\xi_i\to \bar \xi$ because $\pa B_1$ is compact. By stereographic projection with $\xi_i$ as the south pole, equation \eqref{eq:22} is transformed into
 \be
\lda_{as,p_i}(K) K_i(x') u_i(x')^{p_i}= \int_{\R^n_+} P(x',y)\mathcal{P}u_i(y)^{\frac{n+2}{n-2}}\,\ud y \quad \forall~x'\in \R^{n-1},
\ee
where $K_i(x')=K(F(x')) (\frac{2}{|x'|^2+1})^{((n-2)p_i-n)/2}$ and $u_i(x')= (\frac{2}{|x'|^2+1})^{\frac{n-2}{2}} v_i(F(x'))$. Hence $u_i(0)=\max_{\R^{n-1}} u_i \to \infty$ as $i\to \infty$. By Theorem \ref{thm:bp procedure}, we have, subject to a subsequence,
\[
\phi_i =\frac{1}{u_i(0)} u_i(u_i(0)^{p_i-\frac{n+2}{n-2}}x') \to \phi(x') \quad \mbox{in }C_{loc}^{1/2}(\R^{n-1})
\]
for some $\phi\ge 0$ satisfying
\be \label{eq:lim}
\lda K(\bar \xi) \phi (x')^{\frac{n}{n-2}}= \int_{\R^{n}_+} P(x',y) \mathcal{P}\phi(y)^{\frac{n+2}{n-2}}\,\ud y.
\ee
By \cite{HWY}, $\phi$ is classified. Since $v_i $ is nonnegative and antipodally symmetric, for any small $\delta>0$ we have
\begin{align*}
1=\int_{\pa B_1} K v_i^{p_i+1}\,\ud s &\ge 2 \int_{F(B_{\delta}')} K v_i^{p_i+1}\,\ud s
=2  \int_{ B_{\delta}'} K_i u_i^{p_i+1}\,\ud x' \\&
=2 \int_{ B_{\delta u_i(0)^{\frac{n+2}{n-2}-p_i} }'} K_i(u_i(0)^{p_i-\frac{n+2}{n-2}} y') \phi_i(y')^{p_i+1}\,\ud y'\\&
\ge 2 \int_{ B_{R}'} K_i(u_i(0)^{p_i-\frac{n+2}{n-2}} y') \phi_i(y')^{p_i+1}\,\ud y' \to 2 K(\bar \xi) \int_{ B_{R}'}  \phi (y')^{\frac{2(n-1)}{n-2}}\,\ud y'
\end{align*}
as $i \to \infty $ for any fixed $R>0$.
It follows that
\be \label{eq:energybound}
1\ge 2 K(\bar \xi) \int_{\R^{n-1}}  \phi (y')^{\frac{2(n-1)}{n-2}}\,\ud y'.
\ee
Hence, it follows from \eqref{eq:main-ineq}, \eqref{eq:lim} and \eqref{eq:energybound} that
\begin{align*}
S(n)^{\frac{2n}{n-2}} &\ge \frac{\int_{\R^n_+}|\mathcal{P}\phi|^{\frac{2n}{n-2}}}{ (\int_{\R^{n-1}} |\phi|^{\frac{2(n-1)}{n-2}})^{\frac{n}{n-1}}}\\&
= \lda K(\bar \xi)  (\int_{\R^{n-1}} |\phi|^{\frac{2(n-1)}{n-2}})^{-\frac{1}{n-1}}
\ge \lda K(\bar \xi)^{\frac{n}{n-1}} 2^{\frac{1}{n-1}}.
\end{align*}
This yields
\[
\lda \le \frac{S(n)^{\frac{2n}{n-2}}}{(\min K)^{\frac{n}{n-1}} 2^{\frac{1}{n-1}}},
\]
which contradicts the assumption \eqref{eq:strictineq}. We complete the proof.

\end{proof}

\begin{prop} \label{prop:strict} Let $n\ge 3$ and $K\in C^1(\pa B_1)$ be a positive function satisfying $K(\xi)=K(-\xi)$.  For every $q>n-1$, there exists a constant $\delta>0$, depending only on $n$ and $q$, such that if for  a minimal point $\xi_1$ of $K$ there holds $K(\xi)-K(\xi_1)\le \delta |\xi-\xi_1|^q$ for all $\xi\in \pa B_1$,  then \eqref{eq:strictineq} is valid.

\end{prop}

\begin{proof}
Let $\xi_1\in \pa B_1$ be a minimum point of $K$ and $\xi_2=-\xi_1$. Without loss of generality, we may assume $\xi_1$ is the south pole. For $\beta>1$ and $i=1,2$, let
\be
v_{i,\beta}(\xi)=\begin{cases} \left(\frac{\sqrt{\beta^2-1}}{\beta-\cos r_i}\right)^{\frac{n-2}{2}}& \quad \mbox{if }r_i\le \frac{\pi}{2}\\
0& \quad \mbox{if }r_i> \frac{\pi}{2},
 \end{cases}
\ee
and
\[
v_\beta=v_{1,\beta}+v_{2,\beta},
\]
where $r_i=d(\xi,\xi_i)$ is the geodesic distance between $\xi$ and $\xi_i$ on the sphere. Let
\[
u_{i,\lda}(y')= (\frac{2}{1+|y'|^2})^{\frac{n-2}{2}} v_{i,\beta}(F(y')),
\]
where $F(y')$ is the inverse of stereographic projection and $\lda=\sqrt{\frac{\beta-1}{\beta+1}}$, and $u_\lda=u_{1,\lda}+u_{2,\lda}$. By direct computations, we have \[
u_{1,\lda}(y')=2^{\frac{n-2}{2}}\left(\frac{\lda}{\lda^2+|y'|^2}\right)^{\frac{n-2}{2}}\chi_{\{|y'|\le 1\}}=:w_{1,\lda}(y)\chi_{\{|y'|\le 1\}}
\]
\[
u_{2,\lda}(y')=2^{\frac{n-2}{2}}\left(\frac{\lda}{1+\lda^{2}|y'|^2}\right)^{\frac{n-2}{2}}\chi_{\{|y'|\ge 1\}}.
\]
Hence,
\[
u_{\lda}= w_{1,\lda}+w_{2,\lda},
\]
where \[
w_{2,\lda}(y')= 2^{\frac{n-2}{2}}\left(\left(\frac{\lda}{1+\lda^{2}|y'|^2}\right)^{\frac{n-2}{2}}- \left(\frac{\lda}{\lda^2+|y'|^2}\right)^{\frac{n-2}{2}}\right) \chi_{\{|y'|\ge 1\}}.
\]
Since $1+\lda^{2}|y'|^2- (\lda^2+|y'|^2)=(|y|^2-1)(\lda^2-1)<0 $ if $|y'|\ge 1$ and $\lda<1$, we have $w_{2,\lda}(y')\ge0$.
Let $U_{\lda}=\mathcal{P}u_{\lda} =:W_{1,\lda}+W_{2, \lda}$. We have
\[
W_{1,\lda}(y)=2^{\frac{n-2}{2}}\left(\frac{\lda}{(y_n+\lda)^2+|y'|^2}\right)^{\frac{n-2}{2}}.
\]
For $|y|\le \frac{1}{2}$, we have
\begin{align*}
W_{2, \lda}(y)&= \frac{1}{n\w_n} \int_{\R^{n-1}\setminus B_1} \frac{y_n}{(|x'-y'|^2+y_n^2)^\frac{n}{2}} w_{2,\lda}(x')\,\ud x'\\& \ge \frac{1}{C} y_n \int_{\R^{n-1}\setminus B_1} |x'|^{-n} w_{2,\lda}(x')\,\ud x' \\&  \ge \frac{1}{C}\lda^{\frac{n-2}{2}} y_n
\end{align*}
for some $C>0$ independent of $\lda$.

By the conformal invariance, antipodal symmetry and the fact
\[
\mathcal{P} v_\beta(\xi)\le C\lda^{-\frac{n-2}{2}} dist(\xi,\{\xi_1,\xi_2\})^{2-n},
\] we have
\begin{align}
&\int_{B_1} |\mathcal{P}v_\beta|^{\frac{2n}{n-2}}\,\ud \xi= \int_{\R^n_+} U_\lda^{\frac{2n}{n-2}}\,\ud y\nonumber \\&
=2 \int_{B_{1/2}^+} U_\lda^{\frac{2n}{n-2}}\,\ud y + O(\lda^{n})\nonumber \\&
=2 \int_{B_{1/2}^+} W_{1,\lda}^{\frac{2n}{n-2}}+\frac{2n}{n-2}W_{1,\lda}^{\frac{n+2}{n-2}} W_{2,\lda}\,\ud y + O(\lda^{n})
\nonumber  \\& \ge 2 ( \w_n + \frac{2^{\frac{n+4}{2}}n}{n-2}\frac{\lda^{n-1}}{C} \int_{B_{1/2\lda}^+}(\frac{1}{(z_n+1)^2+|z'|^2})^{\frac{n+2}{2}}z_n\,\ud z+O(\lda^{n}))\nonumber  \\&
=2 ( \w_n +\frac{2^{\frac{n+4}{2}}n}{n-2}\frac{\lda^{n-1}}{C} \int_{\R^n_+}(\frac{1}{(z_n+1)^2+|z'|^2})^{\frac{n+2}{2}}z_n\,\ud z+O(\lda^{n})) \nonumber \\&
=:2 ( \w_n +
A\lda^{n-1}+O(\lda^{n}))
\label{eq:one-side}
\end{align}
with $A>0$.
On the other hand, let $q>n-1$ and suppose  $K(\xi)-K(\xi_1)\le \delta |\xi-\xi_1|^q$, where $\delta>0$ is to be fixed.  It follows that
\begin{align*}
&\int_{\pa B_1} K v_\beta^{\frac{2(n-1)}{n-2}}=2\int_{\pa B_1\cap \{x_{n}<0\}}  K v_{1,\beta}^{\frac{2(n-1)}{n-2}}\\&
=2 (K(\xi_1)\int_{\pa B_1\cap \{x_{n}<0\}}   v_{1,\beta}^{\frac{2(n-1)}{n-2}}+\delta \int_{\pa B_1\cap \{x_{n}<0\}}   |\xi-\xi_1|^q v_{1,\beta}^{\frac{2(n-1)}{n-2}})
\nonumber  \\& \le 2K(\xi_1) (n \w_n + \delta C(n,q) \lda^{n-1}).
\label{eq:other-side}
\end{align*}
Setting $\frac{n}{n-1}\delta C(n,q)<A$, for small $\lda$ we have
\begin{align*}
\frac{\int_{B_1} |\mathcal{P}v_\beta|^{\frac{2n}{n-2}}\,\ud \xi}{(\int_{\pa B_1} K v_\beta^{\frac{2(n-1)}{n-2}})^{\frac{n}{n-1}}}
> \frac{2\w_n}{(2K(\xi_1) n \w_n)^{\frac{n}{n-1}}}=\frac{S(n)^{\frac{2n}{n-2}}}{(\min_{\pa B_1} K)^{\frac{n}{n-1}}2^{1/(n-1)}}.
\end{align*}

Therefore, we complete the proof.

\end{proof}

\begin{proof}[Proof of Theorem \ref{thm:A}] It follows immediately from Proposition \ref{prop:existence} and Proposition \ref{prop:strict}.

\end{proof}

\bigskip

\noindent School of Mathematical sciences, Beijing Normal University\\
Beijing 100875, China\\[1mm]
Email: \textsf{jx@bnu.edu.cn}

\end{document}